\documentclass[11pt]{amsart}
\usepackage{amsfonts}
\usepackage{amsfonts,latexsym,rawfonts,amsmath,amssymb,amsthm}
\usepackage[plainpages=false]{hyperref}

\usepackage{graphicx}
\makeatletter
\renewcommand\@biblabel[1]{}
\makeatother

\numberwithin{equation}{section}

\newcommand{\beq}{\begin{equation}}
\newcommand{\eeq}{\end{equation}}
\newcommand{\beqs}{\begin{eqnarray*}}
\newcommand{\eeqs}{\end{eqnarray*}}
\newcommand{\beqn}{\begin{eqnarray}}
\newcommand{\eeqn}{\end{eqnarray}}
\newcommand{\beqa}{\begin{array}}
\newcommand{\eeqa}{\end{array}}

\def\lra{\longrightarrow}

\def\bc{\begin{center}}
\def\ec{\end{center}}

\def\begeq{\begin{equation}}
\def\endeq{\end{equation}}
\def\and{\quad{\rm and}\quad}

\let\lra=\longrightarrow

\def\mapright\#1{\,\smash{\mathop{\lra}\limits^{\#1}}\,}

\newtheorem{prop}{Proposition}[section]
\newtheorem{theo}[prop]{Theorem}
\newtheorem{lem}[prop]{Lemma}
\newtheorem{claim}[prop]{Claim}

\newtheorem{rem}[prop]{Remark}

\newtheorem{defi}[prop]{Definition}

\title{Fano manifolds with weak almost K\"ahler-Ricci solitons}

\author   {Feng Wang }
\author { Xiaohua $\text{Zhu}^*$}

\thanks {* Partially supported by the NSFC Grants 10990013 and 11271022}

\address{ Feng Wang\\School of Mathematical Sciences, Peking University,
Beijing, 100871, China}

\address{ Xiaohua Zhu\\School of Mathematical Sciences and BICMR, Peking University,
Beijing, 100871, China\\
 xhzhu@math.pku.edu.cn}

\begin{document}

\bibliographystyle{plain}

\begin{abstract} In this paper,    we prove that a sequence of
  weak  almost K\"ahler-Ricci solitons  under further suitable conditions   converge to a  K\"ahler-Ricci soliton   with complex codimension of singularities at least 2   in  the Gromov-Hausdorff topology.  As a corollary,   we  show that  on  a Fano manifold with the modified K-energy  bounded below,  there exists  a sequence of  weak  almost K\"ahler-Ricci solitons   which converge to a  K\"ahler-Ricci soliton  with complex codimension of singularities at least 2   in  the Gromov-Hausdorff topology.

\end{abstract}

\date{}

\maketitle

\tableofcontents

\setcounter{section}{-1}
\section{Introduction}

 In  [WZ],    we studied  the  structure of  the  limit space for a sequence  of Riemannian manifolds with  the Bakry-\'{E}mery Ricci curvature bounded
  below  in  the Gromov-Hausdorff topology.   In particular, for a  sequence of
  weak  almost K\"ahler-Ricci solitons  $\{(M_i, g^i, J_i)\}$,  we showed  that there exists a subsequence of   $\{(M_i, g^i,J_i)\}$  which converge to a metric  space $(Y,g_\infty)$   with complex codimension of singularities at least 2    in  the Gromov-Hausdorff topology.   As in [CC] for Riemannian manifolds with  the Ricci curvature bounded below,  each tangent space on  $(Y,g_\infty)$  is a metric cone.   The present  paper is  a continuance of [WZ].   We   further prove the smoothness of  the metric $g_\infty$  on the regular part  $\mathcal{R}$ of $Y$ under further suitable conditions.  Actually,  $g_\infty$ is a  K\"ahler-Ricci soliton on $\mathcal{R}$.

   Inspired by a  recent work of Tian and Wang  on  weak  almost Einstein metrics [TW],  we  use the K\"ahler-Ricci flow
 to smooth the sequence of K\"ahler metrics  $\{(M_i,g^{i}, J_i)\}$ to get  the $C^\infty$-convergence.     To realize this,   we  shall  first establish a  version of Perelman's pseudolocality theorem  for the  Hamilton's Ricci flow with  the Bakerly-\'Emergy Ricci curvature condition,  then we  control the deformation of  distance functions along the K\"ahler-Ricc flow as in [TW].

 It is useful to  mention that   there are two new  ingredients   in  our case compared to [TW]:  One  is that  we  modify the K\"ahler-Ricci flow to derive  an estimate for  the modified Ricci curvature (cf. Section 2); another is that we  estimate the growth of  the $C^0$-norm  of holomorphic vector fields associated to  the K\"ahler-Ricci solitons  along the flow (cf. Section 4).   The late is usually dependent  of  the initial metric $g^i$ of the  K\"ahler-Ricci flow. But for a  family of K\"{a}hler metrics  $g^s$ ($0<s<1$) constructed from   solutions  of a family of
complex Monge-Amp\`ere equations on a Fano manifold with the modified K-energy  bounded below [WZ],  we get a  uniform $C^0$-norm for the holomorphic vector field (cf. Lemma \ref{further-condition-2'}) under the deformed metrics.

The following can be regarded as  the main result in this paper.

 \begin{theo}\label{main-theorem-0}  Let $(M,J)$  be  a Fano manifold with the modified K-energy  bounded below.  Then  there exists  a sequence of
  weak  almost K\"ahler-Ricci solitons   on $(M,J)$  which converge to a  K\"ahler-Ricci soliton  with complex codimension of singularities at least 2   in the Gromov-Hausdorff topology.  In the other words,  a Fano manifold with   the modified K-energy  bounded below can be deformed to a  K\"ahler-Ricci soliton  with complex codimension of singularities at least 2.
\end{theo}

 The organization of paper is as follows.  In Section 1,  we prove  a pseudolocality theorem  of Perelman   for  the Hamilton's Ricci flow  under the Bakry-\'{E}mery  Ricci curvature condition.
  In Section 2,    we focus on  the (modified)  K\"ahler-Ricc flow to give a  local estimate for the Ricci curvature along the flow.
   Section 3 is devoted to estimate   the distance functions along the K\"ahler-Ricc flow.  In  Section 4,   we prove the main theorems in this paper,  Theorem \ref{main-theorem-1} and  Theorem \ref{main-theorem-2} (Theorem \ref{main-theorem-0}).

\section{A version of pseudolocality theorem}

In  this section, we prove a  version of Perelman's pseudolocality theorem  with  Bakerly-\'Emergy Ricci curvature condition (cf. Theorem 11.2 in [Pe]).
A similar version was recently appeared in [TW]. Since our case is lack of the lower bound of scalar curvature, in particular the lower bound of Ricci curvature,
we  will modify the arguments both in [Pe] and [TW].

First we recall a result about an estimate of isoperimetric constant on a geodesic  ball. This result comes  essentially  from a lemma  in [Li] and the volume comparison theorem with  the Bakerly-\'Emergy Ricci curvature bounded  below  in [WW].

\begin{lem}\label{isoperimetrc} Let $(M,g)$  be a Riemannian manifold with
\begin{align}\label{BE-condition}
\text{Ric}({g})+\text {hess }_{g} f\geq -(n-1) c g, ~|\nabla_{g}f|\le A.\end{align}
Then for any geodesic balls in $M$,  $(B_p(s))$,  $(B_p(r))$ with $r\ge s$,   there exists a uniform $C=C(n)$ such that
\begin{align}
ID_{\frac{n}{n-1}}(B_p(s))\geq C^{\frac{1}{n}}(\frac{\text {vol}(B_p(r))-\text {vol}(B_p(s))}{v(r+s)})^{\frac{n+1}{n}},
\end{align}
where $v(r)=e^{2Ar}\text{vol}_{c}(r)$ and $\text{vol}_{c}(r)$ denotes the volume of $r$-geodesic ball in the space form with constant curvature $-c$.
\end{lem}

Lemma \ref{isoperimetrc} will be used to get a uniform Sobolev constant in  the proof of  following pseudolocality theorem  in  the Bakerly-\'Emergy geometry.

\begin{theo}\label{pseudolocality-theorem}
For any $\alpha,r\in [0,1]$, there exist $\tau=\tau(n,\alpha), \eta=\eta(n,\alpha),\epsilon=\epsilon(n,\alpha), \delta=\delta(n,\alpha),$
such that if $(M^n,g(\cdot, t))$ $(0\leq t\leq (\epsilon r)^2)$ is a solution of  Ricci flow,
\begin{align}\label{hamilton-ricci-flow}
\frac{\partial g}{\partial t}=-2\text{Ric}(g),
\end{align}
whose initial metric $g(\cdot,0)=g_0$  satisfies
\begin{align}\label{pseudo-condition}
\text{Ric}(g_0)+\text {hess }_{g_0} f\geq -(n-1)r^{-2}\tau^2 g_0, ~|\nabla f|_{g_0}\leq r^{-1}\eta, \end{align}
 and
  \begin{align}\label{maximal-volume-0} \text{vol}(B_q(r,g_0))\geq (1-\delta)c_n r^n,
\end{align}
where $c_n$ is the volume of unit ball in the Euclidean space $\mathbb R^n$, then for any $x\in B_q(\epsilon r,g_0)$ and $t\in(0,(\epsilon r)^2]$, we have
\begin{align}\label{peudo-cuvature}
|Rm(x,t)|< \alpha t^{-1}+(\epsilon r)^{-2},
\end{align}
Moreover,
\begin{align}\label{volume-estimate}\text{vol }B_x(\sqrt{t})\geq \kappa(n)t^{\frac{n}{2}},
\end{align}
where $\kappa(n)$ is a uniform constant.

\end{theo}

\begin{proof}
 By  scaling the metric, we may assume $r=1$ in the theorem.   As in   [Pe], we  use   the argument    by contradiction to prove  (\ref{peudo-cuvature}).   On contrary,  we
 suppose that for some $\alpha>0$, there are $\tau_i, \eta_i,\delta_i, \epsilon_i$ which  approaching zero as $i\to\infty$, and there are  a sequence of manifolds
 $\{(M_i,g^i)\}$  which satisfying (\ref{pseudo-condition}) and  (\ref{maximal-volume-0}) with some points $q_i\in M_i$   such that (\ref{peudo-cuvature}) doesn't   hold at $( x_i, \bar t_i)$  for   some points $x_i\in B_{q_i}(\epsilon_i ,g^i_0)$  some time $\bar t_i\le\epsilon_i^2 $ along   the Ricci flows  $(M_i,  g^i_t= g^i(\cdot,  t))$  with  $g^i=g^i_0$ as the initial metrics.
Without the loss of generality, we may also assume that
\begin{align}\label{peudo-cuvature-modified}
|Rm(x,t)|\leq \alpha t^{-1}+ (\epsilon_i )^{-2}, ~\forall~t\in (0,\overline t_i], x\in B_{q_i}(\epsilon_i, g^i_0).
\end{align}
 Then as showed in [Pe],  for any $A<\frac{1}{100n \epsilon_i}$,  there exist  points $(\bar{x}_i,\bar{t}_i)$ such that  for any $(x,t)$ with
 $$\bar{t}_i-\frac{1}{2}\alpha Q^{-1}\leq t\leq \bar{t}_i, d_{g_t^i}(x,\bar{x}_i)\leq\frac{1}{10}AQ^{-\frac{1}{2}},$$
\begin{align}\label{curvature-blow-up}
|Rm(x,t)|\leq 4Q,
\end{align}
where $Q=|Rm(\bar{x}_i,\bar{t}_i)|\to\infty.$

Now we consider a solution $u_i(x,t)=(4\pi(\bar{t}_i-t))^{-\frac{n}{2}}e^{-p_i(t,x)}$ of  the conjugate heat equation associated to  the flow $(M_i,  g^i_t)$
 which starts  from a delta function $\delta(\bar{x}_i,\bar{t}_i)$. Namely,   $u_i(x,t)$ satisfies
$$\square^*u_i(x,t)=(-\frac{\partial}{\partial t}-\Delta+R) u_i(x,t)=0,$$
 where $R=R(\cdot,t)$ is the scalar curvature of $g^i_t$. Then the  function
$$v_i(x,t)=[(\bar{t}_i-t)(2\Delta p_i-|\nabla p_i|^2+R)+p_i-n]u_i$$
is nonpositive.
 Moreover,
 there exists  a positive constant $\beta$ such that
 \begin{align}\label{beta-bound}
 \int_{B_{\bar{x}_i}(\sqrt{\bar{t}_i-\tilde{t}_i}, g^i_{\tilde t_i})} v_i\leq -\beta,\end{align}
  for some $\tilde{t}_i\in[\bar{t}-\frac{1}{2}\alpha Q^{-1},\bar{t}_i]$,  when $i$ is large enough [Pe].

Let $\phi$ be  a cut-off function which is equal to 1 on [0,1] and decreases to 0 on [1,2].  Moreover, it satisfies $\phi''\geq -10\phi, (\phi')^2\leq 10\phi.$   Putting $h_i=\phi(\frac{\tilde{d}_i(x,t)}{10 A\sqrt{\bar{t}_i}})$, where $\tilde{d}_i(x,t)=d_{g^i_t}(\bar{x}_i, x)+200 n\sqrt{t}$.
Then by Lemma 8.3 in [Pe] with the help of  (\ref{peudo-cuvature-modified}), we get
\begin{align}
\nonumber (\frac{\partial}{\partial t}-\Delta) h_i &=\frac{1}{10A\sqrt{\bar{t_i}}}(d_t-\Delta d+ \frac{100n}{\sqrt{\bar{t}}}) \phi'-(\frac{1}{10A\sqrt{\bar{t_i}}})^2\phi''\\
&\leq (\frac{1}{10A\sqrt{\bar{t_i}}})^2 10\phi,  ~\forall~t\in (0,\overline t_i], x\in B_q(\epsilon, g^i_0).\notag
\end{align}
 It follows
\begin{align}
\frac{d}{dt}\int_{M_i}(-h_iv_i)&=\int_{M_i}\square h_i(- v_i)+\int_{M_i}h_i\square^* v_i\notag\\
&\le -\frac{1}{100A^2\bar{t}_i}\int_{M_i} h_i v_i,\notag
\end{align}
where $\square= \frac{\partial}{\partial t}-\Delta$ and  we used the fact  that $\square^* v_i\le 0$ [Pe].  Thus by (\ref{beta-bound}), we obtain
\begin{align}\label{hv}\beta(1-A^{-2})\leq-\int_{M_i} ( h_iv_i)(0,\cdot).\end{align}
Similarly, we can show
\begin{align}\label{normalization}
\int_{M_i} (\hat {h}_iu_i )(\cdot,0)\ge 1-4A^{-2},
\end{align}
 where  $\hat {h}_i=\phi(\frac{\tilde{d}_i(x,t)}{5 A\sqrt{\bar{t}_i}})$. The above implies that
\begin{align}\label{hu}
\int_{B_{\bar{x}_i}(20A\sqrt{\bar{t}_i})\setminus B_{\bar{x}_i}(10A\sqrt{\bar{t}_i})}u_i(\cdot, 0)\leq 1-\int_{M_i} \hat {h}_i u_i \leq4A^{-2}.
\end{align}

On the other hand,   by (\ref{hu}),   we see that
\begin{align}
&-\int_{M_i} (h_iv_i)=\int_M[\bar{t}_i(-2\Delta p_i+|\nabla p_i|^2-R)-p_i+n]h_i u_i\notag\\
&=\int_{M_i}[-\bar{t}_i|\nabla \tilde{p}_i|^2-\tilde{p}_i+n]\tilde{u}_i+\int_{M_i}[\bar{t}_i(\frac{|\nabla h_i|^2}{h_i}-Rh_i)-h_i\ln h_i]u_i
\notag\\
&\leq \int_{M_i}[-\bar{t}_i|\nabla \tilde{p}_i|^2-\tilde{p}_i+n]\tilde{u}_i-\bar{t}_i\int_{M_i}R\tilde{u}_i+A^{-2}+100\epsilon^2,\notag
\end{align}
 where $\tilde{u}_i=h_iu_i$ and $\tilde{p}_i=p_i-\ln h_i$.
Note that  $R+\Delta f_i\geq -(n-1)$.   Then
\begin{align}
\nonumber -\int_{M_i}R\tilde{u}_i\leq n-1-\int_{M_i}\langle \nabla f_i,\tilde{u}_i\rangle \leq n-1+\eta_i\int_{M_i}|\nabla \tilde{u}_i|\\
\nonumber \leq n-1+\eta_i\sqrt{\int_{M_i}|\nabla \tilde{p}_i|^2\tilde{u}_i}\leq n+\eta_i\int_{M_i}|\nabla \tilde{p}_i|^2\tilde{u}_i.
\end{align}
Hence, by (\ref{hv}),  we get
$$\int_{M_i}[-\bar{t}_i(1-\eta_i)|\nabla \tilde{p}_i|^2-\tilde{p}_i+n]\tilde{u}_i\ge \beta(1-A^{-2}) -(100+n)\epsilon^2-A^{-2}.$$
Therefore,  by  rescaling  these metrics $g^i_0$  to  $ \widehat {g^i_0}=\frac{1}{2}[\bar{t}_i(1-\eta_i)]^{-1} g^i_0$,  we derive
\begin{align}\label{log-sobolev}
\int_{B_{\bar{x}_i}(20A)}[-\frac{1}{2}|\nabla \tilde{p}_i|^2-\tilde{p}_i+n]\widehat{u}_i\ge (1-\eta_i)^{\frac{n}{2}} \mu>\mu_0>0,
\end{align}
where   $\widehat{u}_i=(2\pi)^{-\frac{n}{2}}e^{-\tilde{p}_i}$ and  $\mu=\beta(1-A^{-2})-(100+n)\epsilon^2-A^{-2}$.
Normalize   $\widehat{u}_i$ by multiplying a constant $c$  so that
$$\int_{B_{\bar{x}_i}(20A)}c\widehat{u}_i=1.$$
 By (\ref{normalization}),  it is easy to  see   that (\ref{log-sobolev})  still holds for the normalized  $\widehat{u}_i$.

Next as in [TW].  we  introduce a  functional
$$F_i(u)=\int_{B_{\bar{x}_i}(20A)}(2|\nabla u|^2-2u^2\log u-n(1+\log\sqrt{2\pi})u^2),$$
 defined  for  any  nonnegative functions $u\in W_0^{1,2}(B_{\bar{x}_i}(20A), \widehat {g^i_0} )$ with
 $$\int_{B_{\bar{x}_i}(20A)}u^2=1.$$
Clearly,  by (\ref{log-sobolev}),   we have
 \begin{align}\label{low-bound-u}\lambda_i\leq F_i(\sqrt{c\widehat{u}_i})\leq -\mu_0<0,\end{align}
where $\lambda_i=\inf_{u\in W_0^{1,2}(B_{\bar{x}_i}(20A), \widehat {g^i_0} )} F_i(u)$.
According to [Ro],   the infinity of $ F_i(u)$ can be achieved by a  minimizer $\phi_i$ which   satisfies
 the Euler-Lagrange equation  on $(B_{\bar{x_i}}(20A), \widehat {g^i_0})$,
 \begin{align}\label{el-equation}
-2\Delta \phi_i(x)-2\phi_i(x)\log\phi_i(x)-n(1+\log\sqrt{2\pi})\phi_i(x)=\lambda_i \phi_i(x).
\end{align}

We need  to estimate the $L^\infty$-norms  and gradient norms of those $\phi_i$.
Note that  $\log x\leq \frac{n}{2}x^{\frac{2}{n}}$.  Then
\begin{align}&\lambda_i+n(1+\log\sqrt{2\pi})\notag\\
\nonumber & =2\int_{B_{\bar{x}_i}(20A)}|\nabla \phi_i(x)|^2-2\int_{B_{\bar{x}_i}(20A)}\phi_i(x)^2\log\phi_i(x)\\
\nonumber &\geq 2\int_{B_{\bar{x}_i}(20A)}|\nabla \phi_i(x)|^2-n \int_{B_{\bar{x}_i}(20A)}\phi_i(x)\phi_i(x)^{\frac{n+2}{n}}\\
\nonumber &\geq 2\int_{B_{\bar{x}_i}(20A)}|\nabla \phi_i(x)|^2-n (\int_{B_{\bar{x}_i}(20A)}\phi_i(x)^{\frac{2n}{n-2}})^{\frac{n-2}{2n}}\\
\nonumber &\geq  2\int_{B_{\bar{x}_i}(20A)}|\nabla \phi_i(x)|^2-\{a^2 (\int_{B_{\bar{x}_i}(20A)}\phi_i(x)^{\frac{2n}{n-2}})^{\frac{n-2}{2n}}+\frac{n^2}{4a^2}\}.
\end{align}
Since   the  Sobolev constants  $C_S$ are uniformly bounded below on $(B_{\bar{x}}(\frac{1}{2})$, $g^i_0)$   according to Lemma  \ref{isoperimetrc},
 by choosing  the number $a$ small enough,  we see that $\lambda_i$  is uniformly bounded below and
   $\int_{B_{\bar{x}_i}(20A)}|\nabla \phi_i(x)|^2$ is uniformly bounded.
Applying  the standard Moser iteration method  to (\ref{el-equation}),  we  will get
\begin{align}\label{C^0-norm}
|\phi_i(x)|< C_1(\mu_0, n, C_S).
\end{align}
As a consequence, $\phi_i(x)$ is  an almost sub-solution of the  Lapalace equation.  Hence,   we can also get a uniform oscillation estimate for $\phi_i(x)$ near
the  boundary of $B_{\bar{x}_i}(20A)$.  In fact, as in [WT],  we can show that for any $w\in \partial B_{\bar{x}_i}(20A)$
\begin{align}\label{osc-boundary} Osc_{B_w(2^{-N})}(\phi_i)< C\gamma^{N-1}+\frac{\gamma^{N-1}-4^{-N+1}}{4(4\gamma-1)}, \end{align}
for some uniform $C$,  where $N\ge 2$ is  any  integer and    the number $\gamma$  can be chosen  in  the interval  $(\frac{1}{2},1)$.

To get  the interior  gradient estimate for $\phi_i(x)$,  we will also use the Moser iteration method. For simplicity,  we let $\phi=\phi_i$ for each $i$.   First we note that by (\ref{el-equation}) and the estimate (\ref{C^0-norm}),  it holds
\begin{align}
\langle \nabla \phi,\nabla \Delta \phi\rangle\geq -C_2(\nu, n, C_S)|\nabla \phi(x)|^2.
\end{align}
Then by   the Bochner identity,
 $$\frac{1}{2}\Delta|\nabla \phi|^2=|\text{ hess }\phi|^2+R_{ij}\phi_i\phi_j+\langle\nabla\phi,\nabla \Delta \phi\rangle,$$
 we obtain
  \begin{align}\label{inequality}
\frac{1}{2}\Delta|\nabla \phi|^2\geq |\text{ hess }\phi|^2-f_{ij}\phi_i\phi_j-(C_2+(n-1)\tau^2)|\nabla \phi|^2.
\end{align}

Let
 $\rho$  be a cut-off function on  the interval  $[0, 20A]$ which is supported in a subset of $[0, 20As)$, where $s<1$.  Then
multiplying both sides of (\ref{inequality}) by $\rho(d(\bar{x}_i,.) w^{p}$,  where  $w=|\nabla \phi|^2$ and $p\ge 0$, we get
\begin{align}\label{gradient-integral-1}
\nonumber &\frac{2p}{(p+1)^2}\int_{B_{\bar{x}_i}(20A)}\rho(d(\bar{x}_i,.))|\nabla w^{\frac{p+1}{2}}|^2&\\
\nonumber &=\frac{1}{2}\int_{B_{\bar{x}_i}(20A)}\rho(d(\bar{x}_i,.))(-\Delta w)w^p-\frac{1}{2}\int_{B_{\bar{x}_i}(20A)}\langle\nabla \rho(d(\bar{x}_i,.)),\nabla w\rangle w^p&\\
\nonumber &\leq\int_{B_{\bar{x}_i}(20A)}\rho(d(\bar{x}_i,.))w^p (- |\text{ hess }\phi|^2+f_{ij}\phi_i\phi_j+C_2|\nabla \phi|^2)\notag\\
&-\frac{1}{2} \int_{B_{\bar{x}_i}(20A)}w^p \langle\nabla \rho(d(\bar{x}_i,.)),\nabla w\rangle.
\end{align}
On the other hand,    using the integration  by parts,  we have
\begin{align}
\nonumber &\int_{B_{\bar{x}_i}(20A)}\rho(d(\bar{x}_i,.)) w^pf_{ij}\phi_i\phi_j\notag\\
&=-\int_{B_{\bar{x_i}}(20A)}\rho_lf_l\phi_i\phi_jw^p+\int_{B_{\bar{x}_i}(20A)}\rho f_j\phi_{ij}\phi_jw^p\notag\\
\nonumber &-p\int_{B_{\bar{x}_i}(20A)}\rho w^{p-1}f_iw_j\phi_i\phi_j-\int_{B_{\bar{x}_i}(20A)}\rho w^p f_i\phi_i\Delta \phi.&
\end{align}
Observe that
\begin{align}
\nonumber|\int_{B_{\bar{x}_i}(20A)}\rho_lf_l\phi_i\phi_jw^p|&\leq \eta\int_{B_{\bar{x}_i}(20A)}|\rho'|w^{p+1},&\\
\nonumber|\int_{B_{\bar{x}_i}(20A)}\rho f_j\phi_{ij}\phi_jw^p|&\leq 2\eta\int_{B_{\bar{x}_i}(20A)}\rho(|\text{ hess }\phi|^2+w)w^p,&\\
\nonumber|\int_{B_{\bar{x}_i}(20A)}\rho w^{p-1}f_iw_j\phi_i\phi_j|&\leq 2\eta(\int_{B_{\bar{x}_i}(20A)}\rho w^{p-1}|\nabla w|^2+ \int_{B_{\bar{x}_i}(20A)}\rho w^{p+1}),&\\
\nonumber|\int_{B_{\bar{x}_i}(20A)}\rho w^p f_i\phi_i\Delta \phi|&\leq   C_3(\nu_0, n, C_s)\eta\int_{B_{\bar{x}_i}(20A)}\rho w^{p+\frac{1}{2}}.&
\end{align}
Hence,  by (\ref{gradient-integral-1}), we get
\begin{align}
\nonumber &\frac{p}{(p+1)^2}\int_{B_{\bar{x}_i}(20A)}\rho(d(\bar{x}_i,.))|\nabla w^{\frac{p+1}{2}}|^2&\notag\\
&\le C_4\int_{B_{\bar{x}_i}(20A)} (\rho+ p\eta\rho+ \rho') w^{p+1} +C_5
\int_{B_{\bar{x}_i}(20A)} \eta\rho w^{p+\frac{1}{2}}.\notag
\end{align}
Since we may assume that $w\ge 1$, we deduce
\begin{align}\label{gradient-integral-2}
&\frac{p}{(p+1)^2}\int_{B_{\bar{x}_i}(20A)}\rho(d(\bar{x},.))|\nabla w^{\frac{p+1}{2}}|^2\notag\\
& \le C_5'\int_{B_{\bar{x}_i}(20A)} (\rho+ p\eta\rho+ \rho') w^{p+1} , ~\forall ~p\ge 0.
\end{align}
Note that  the  Sobolev constants  are uniformly bounded below on $(B_{\bar{x}_i}(20A), $ $\widehat{ g^i_0})$. Therefore,  by  choosing
the suitable cut-off functions $\eta$ in (\ref{gradient-integral-2}), we  use
 the iteration method  to derive
\begin{align}\label{gradient-esti} \|\nabla\phi_i\|_{C^0(B_{\bar{x}_i}(20sA))}^2\leq  C_6(1+\int_{B_{\bar{x}_i}
(20A)}|\nabla\phi_i|^2)< C.
\end{align}

It remains to analyze the limit of $\phi_i$.  According to  Corollary 4.8 in [WZ],  we see that $(M_i, \widehat {g^i_0})$ converge to the euclidean space $\mathbb{R}^n$ in
the  Gromov-Hausdorff topology.  Thus  by the estimates (\ref{C^0-norm}) and (\ref{gradient-esti}), there exists a subsequence
of $\phi_i$ which converge to a continuous  limit  $\phi_\infty\ge 0$ on the standard  $B_0(20A) \subset \mathbb{R}^n$.

\begin{claim}\label{claim-1}  $\phi_\infty$ is a solution  of the following equation on  $B_0(20A)$,
\begin{align}\label{limit-equation-euclidean}
-2\Delta \phi_\infty-2\phi_\infty\log\phi_\infty-n(1+\log\sqrt{2\pi})\phi_\infty=\lambda_\infty \phi_\infty,
\end{align}
where $\lambda_\infty<0$.
\end{claim}

 As in [TW], to prove (\ref{limit-equation-euclidean}),  it suffices to show that
 \begin{align}
 -\phi_\infty=\int_{B_0(20A)}G(z,y)(\frac{\lambda_\infty+n(1+\log\sqrt{2\pi})}{2}+\log\phi_\infty)\log\phi_\infty.
 \end{align}
Here  $G(z,y)$ is the Green function  on the ball $B_0(20A)$, which is given by
  $$G(z,y)=\frac{1}{(n-2)nc_n} (d^{2-n}(z,y)-d^{2-n}(0,z)d^{2-n}(z^*,y)),$$
where $z^*$ is the conjugate point of $z$.

Choose  a sequence $z_k\rightarrow z, z_k^*\rightarrow z^*$.   By the  Laplacian comparison for the distance functions on  $(B_{\bar{x_i}}(20A), \widehat {g^i_0})$ [WW],
 \begin{align}
&\Delta  d(z_k,.) \leq(n-1)\tau_i\coth\tau_i d(z_k,.)+ 2\eta_i\notag\\
&\leq \frac{n-1}{d(z_k,.)}+(n-1)\tau_i+2\eta_i, \notag\end{align}
we have
\begin{align}\Delta  d^{2-n}(z_k,.)+(n-2)d^{1-n}(z_k,.)((n-1)\tau_i+2\eta_i)\geq 0.\notag
\end{align}
It follows
\begin{align}
 &\int_{B_{z_k}(20A)\setminus{z_k}}|\Delta  d^{2-n}(z_k,.)|\notag\\
&\leq \int_{B_{z_k}(20A)\setminus\{z_k\}}|\Delta  d^{2-n}(z_k,.)+(n-2)d^{1-n}(z_k,.)((n-1)\tau_i+2\eta_i)|\notag\\
&+\int_{B_{z_k}(20A)\setminus\{z_k\}}(n-2)d^{1-n}(z_k,.)((n-1)\tau_i+2\eta_i).&
\end{align}
By  a direct computation,   we obtain
$$\int_{B_{z_k}(20A)\setminus\{z_k\}}|\Delta  d^{2-n}(z_k,.)|\rightarrow 0,~\text{as}~k\to \infty.$$
Note that
\begin{align}
 &\int_{B_{z_k}(20A)}d^{2-n}(z_k,y)\Delta\phi_k(y)=\notag\\
& (n-2)nc_n\phi_k(z_k)+\int_{B_{x_k}(20A)\setminus\{z_k\}}\phi_k(y)\Delta  d^{2-n}(z_k,y).\notag
\end{align}
Hence we derive that
\begin{align}\label{point-in}
\lim_{k \rightarrow\infty}\int_{B_{z_k}(20A)}d^{2-n}(z_k,y)\Delta\phi_k(y)=(n-2)nc_n\phi_\infty(z).
\end{align}
Similarly,  since   $z_k^*$ is outside $B_{x_k}(20A)$,  we have
\begin{align}\label{point-out}
\lim_{k \rightarrow\infty}\int_{B_{z_k}(20A)}d^{2-n}(z_k^*,y)\Delta\phi_k(y)=0.
\end{align}

Combining (\ref{point-in}) and (\ref{point-out}),  we get
\begin{align}
&-\phi_\infty(z)\notag\\
&=-\lim_{k \rightarrow\infty}\int_{B_{z_k}(20A)}(d^{2-n}(z_k,y)-d^{2-n}(x_k,z_k)d^{2-n}(z_k^*,y))\Delta\phi_k(y)\notag\\
&=\lim_{k \rightarrow\infty}\int_{B_{z_k}(20A)}(d^{2-n}(z_k,y)-d^{2-n}(x_k,z_k)d^{2-n}(z_k^*,y))\notag\\
&\times
(\frac{\lambda_k+n(1+\log\sqrt{2\pi})}{2}+\log\phi_k)\phi_k\notag\\
&=\int_{B_0(20A)}G(z,y)(\frac{\lambda_\infty+n(1+\log\sqrt{2\pi})}{2}+\log\phi_\infty)\phi_\infty.\notag
\end{align}
The claim is proved.

 By  the estimates (\ref{osc-boundary}),   $\phi_\infty$ is  in fact in  $C_0(B_0(20A))$.  Thus  by (\ref{limit-equation-euclidean}), we get
$$F(\phi_\infty)=\int_{B_0(20A)}(2|\nabla \phi_\infty|^2-2\phi_\infty^2\log \phi_\infty-n(1+\log\sqrt{2\pi})\phi_\infty^2)=
\lambda_\infty<0,$$
which  is a contradiction to the Log-Sobolev inequality in  $\mathbb{R}^n$ [Gr].  The proof of (\ref{peudo-cuvature}) is completed.
\end{proof}

To obtain  (\ref{volume-estimate}),  it suffices  to  estimate  the lower bound of the injective radius at $x$.  This  can  be done by using the  same blowing-up
argument as in the proof of (\ref{peudo-cuvature}) (cf. [Pe], [TW]). We leave it to the readers.

\vskip3mm

\section{A Ricci curvature estimate}

In this section, we prove several  technical lemmas which will be used in next sections.   From now on   we assume that $M$  is an $n$-dimensional  Fano manifold with a reductive  holomorphic vector field $X$ [TZ1]. As in [TZ3], we consider the following modified K\"ahler-Ricci flow,
\begin{align}\label{kr-flow}  \frac{\partial}{\partial t} g=-\text{Ric} (g)+ g+ L_X g,\end{align}
with a $K_X$-invariant initial K\"ahler metric $g_0$  in  $2\pi c_1(M)$, where $K_X$ is the one-parameter compact subgroup
generated by  $\text{im}(X)$. Thus  $L_Xg$ is a real valued complex hessian tensor. If we scale $g_0$ by  $\frac{1}{\lambda}$, where $0<\lambda\leq  1$,
then  (\ref{kr-flow}) becomes
\begin{align}\label{rescaling-equation-1} \frac{\partial}{\partial t} g=-\text{Ric}(g)+\lambda g+\lambda L_X g.
\end{align}
Clearly,  the flow is solvable for any $t>0$ and $\omega_{g_t}\in \frac{2\pi}{\lambda}c_1(M)$, where $g_t=g(\cdot,t)$.

By a direct computation from the flow (\ref{rescaling-equation-1}), we  see that
\begin{align}\label{evol1}
 \frac{\partial}{\partial t}R_{i\bar{j}}=&\Delta R_{i\bar{j}}-R_{i\bar{k}}R_{k\bar{j}}+R_{l\bar{k}}R_{i\bar{j}k\bar{l}}\notag\\
  &-\lambda\Delta \theta_{i\bar{j}}+\frac{\lambda}{2}(R_{i\bar{k}}\theta_{k\bar{j}}+R_{k\bar{j}}\theta_{i\bar{k}})-\lambda R_{i\bar{j}k\bar{l}}\theta_{l\bar{k}}
\end{align}
  and
\begin{align}
\frac{\partial}{\partial t}\theta_{i\bar{j}}&=L_X(-\text{ Ric}(g)+\lambda g+\lambda L_X g),\notag\end{align}
where $\theta=\theta_{g_t}$  is a  potential of $X$ associated to $g_t$ such that  $\theta_{i\bar{j}}=L_X g_t$.
Thus if we let  $H=\text{Ric}(g)-\lambda g-\lambda L_X g$,  then we have
\begin{align}\label{ricci-flow-2}
\frac{\partial}{\partial t}H=\Delta H+ \lambda L_X H+ \Lambda (H,Rm),
\end{align}
where $\Lambda$ is a linear operator with bounded coefficients with respect to  the metric $g_t$ and $Rm=Rm(\cdot,t)$ is the sectional curvature of $g_t$.

Moreover,  we have

\begin{lem}\label{evolution-R}
\begin{align}&\frac{\partial}{\partial t}(R-\lambda\Delta\theta-n\lambda)\notag\\
&=\lambda(R-\lambda\Delta\theta-n\lambda)+\Delta(R-\lambda\Delta\theta-n\lambda)-\lambda\Delta \frac{\partial}{\partial t}\theta\notag\\
&+|\text{Ric}(g)-\lambda g-\lambda\sqrt{-1}\partial\bar{\partial}\theta|^2.
\end{align}
\end{lem}

The following lemma is a  consequence of  Theorem \ref{pseudolocality-theorem} in Section 1.

\begin{lem}\label{modified-pseudo} Let $g=g_t$ be a solution of (\ref{rescaling-equation-1}) with  $\omega_{g_0}\in
\frac{2\pi}{\lambda}c_1(M)$.  Suppose that  there exists a small $\delta\le \delta_0<<1$ such that
$g_0$ satisfies:
\begin{align}  &i)~\text{Ric}(g_0)+\lambda L_X{g_0} \geq -(n-1)\delta^2 g_0;\notag\\
&ii)~|X|_{g_0}(x)\leq \frac{\delta}{\lambda}, \forall~ x\in B_q(1,g_0);\notag \\
&iii)~\text{vol}(B_q(1,g_0))\geq (1-\delta)c_n. \notag
\end{align}
Then
 $$
|Rm(x,t)|\leq   4 t^{-1}, \forall ~x\in B_q(\frac{3}{4},g_0),~ t\in (0,2\delta]$$
and
$$
\text{vol}(B_x(\sqrt{t},g(t)))\geq \kappa(n)t^n, $$
where $\kappa=\kappa(n)$  is a uniform constant.
\end{lem}

By Lemma \ref{evolution-R} and Lemma \ref{modified-pseudo}, we prove

\begin{lem}\label{R-estimate} Let $g=g_t$ be a solution of (\ref{rescaling-equation-1}) with  $\omega_{g_0}\in
\frac{2\pi}{\lambda}c_1(M)$.
 Suppose that   for any $t\in [-2,1]$ (we may replace $t$ by $t-3$), $g_t$ satisfies:
  \begin{align} &i)~\text{inj}(q,g_t)\geq 1;\notag\\
&ii)~|Rm(x,t)|\leq 1~\text{and}~
|X|_{g_t}\leq \frac{A}{\lambda},~\forall~x \in B_{q}(1,g_t). \notag
\end{align}
Then
\begin{align}
&|\text{Ric}(g)-\lambda g-\lambda L_X g|(q,0)\notag\\
&\leq C(A,n)\{\int_{-2}^1 dt\int_M|R-n\lambda-\Delta\theta| \omega_{g_t}^n\}^{\frac{1}{2}}.
\end{align}
\end{lem}

\begin{proof}

Putting $h=|H|$, by (\ref{ricci-flow-2}),  we get
\begin{align}
(\frac{\partial}{\partial t}-\Delta)h\le\frac{\Lambda_1(H,H,Rm)}{h}+\lambda X(h)+\lambda \frac{\Lambda_2(\sqrt{-1}\partial\bar{\partial}\theta,H,H)}{h},
\end{align}
where $\Lambda_1,\Lambda_2$ are two linear operators with bounded coefficients with respect to  the metric $g_t$.
Note that under the conditions i) and ii) in the lemma the  Sobolev constants   are uniformly bounded below on $B_{q}(\frac{1}{2}, g_0)$.    Then using the method of Moser iteration, we obtain
\begin{align}\label{moser-iteration}
&|\text{ Ric}(g)-\lambda g-\lambda L_X g|(q,0)\notag\\
&\leq C(A,n)\{\int_{-1}^0 dt\int_M|\text{ Ric}(g)-\lambda g-\lambda L_X g|^2  \omega_{g_t}^n\}^{\frac{1}{2}}.
\end{align}
On the other hand,  we see that there exist   some $t_1\in [-2,-1]$ and  $t_2\in [0,1]$ such that
\begin{align}
\int_M|R-\lambda\Delta\theta-n\lambda| \omega_{g_{t_1}}^n \leq \int_{-2}^{-1} dt\int_M|R-\lambda\Delta\theta-n\lambda|  \omega_{g_t}^n,\notag\\
\int_M|R-\lambda\Delta\theta-n\lambda| \omega_{g_{t_2}}^n\leq \int_{0}^{1} dt\int_M|R-\lambda\Delta\theta-n\lambda|  \omega_{g_t}^n.\notag
\end{align}
Then integrating (\ref{evolution-R}) in  Lemma \ref{evolution-R},  it follows
\begin{align}
 &\int_{t_1}^{t_2} dt\int_M|\text{ Ric}(g)-\lambda g-\lambda L_X g|^2  \omega_{g_t}^n\notag\\
 &\le\int_{t_1}^{t_2} dt\int_M|R-\lambda\Delta\theta-n\lambda|  \omega_{g_t}^n\notag\\
&+\int_M|R-\lambda\Delta\theta-n\lambda| \omega_{g_{t_1}}^n+\int_M|R-\lambda\Delta\theta-n\lambda| \omega_{g_{t_2}}^n\notag\\
 &\leq 3\int_{-2}^{1} dt\int_M|R-\lambda\Delta\theta-n\lambda|  \omega_{g_t}^n.
\end{align}
Hence  by  (\ref{moser-iteration}), we  derive
\begin{align}
h(q,0) &\leq  C(A,m)\{\int_{t_1}^{t_2} dt\int_M|R-\lambda\Delta\theta-n\lambda| \omega_{g_t}^n\}^{\frac{1}{2}}\notag\\
&\le 3C(A,m)\{\int_{-2}^{1} dt\int_M|R-\lambda\Delta\theta-n\lambda| \omega_{g_t}^n\}^{\frac{1}{2}}.\notag
\end{align}
\end{proof}

\begin{lem}\label{tensor}
Under the conditions of Lemma \ref{modified-pseudo}  and $|X|_{g_t}\leq \frac{A}{\lambda\sqrt{t}}$, we have
\begin{align}\label{ricci-estimate}
&|\text{Ric}(g)-\lambda g-\lambda L_X g|(x,s)\notag\\
&\leq C(n,A)s^{-\frac{n+2}{2}}\{\int_0^{2s} dt\int|R-n\lambda-\lambda \Delta \theta|  \omega_{g_t}^n\}^{\frac{1}{2}},
\end{align}
for $0<s\leq \delta$.
\end{lem}

\begin{proof}
By  Lemma \ref{modified-pseudo},  we know that for  $ x\in B_q(\frac{3}{4},g_0)$ and  $t \in (0,2\delta]$,
\begin{align} |Rm(x,t)|\leq t^{-1}~\text{and}~\text{vol} (B_x(\sqrt{t}))\geq \kappa(n)t^n.\notag
\end{align}
Then  the injective radius  estimate in [CGT] implies that
\begin{align}
\nonumber inj(x,t)\geq \xi(n)\sqrt{t}.
\end{align}
Let  $l=\xi(n)^{-1}s^{-\frac{1}{2}}$.   By scaling the metric  $g_t$ as
\begin{align}
\tilde{g}_t=l^{2} g (l^{-2}t+s), t\in[-2,1], \notag
\end{align}
 $\tilde{g}_t$ satisfies
\begin{align}
\nonumber \frac{\partial}{\partial t}\tilde{g}=-\text{Ric}(\tilde{g})+\frac{\lambda}{l^2}\tilde{g}+\frac{\lambda}{l^2}L_X\tilde{g}.
\end{align}
Moreover, $\tilde{g}_t$ satisfies the conditions  i) and ii) in Lemma \ref{R-estimate} for any $t\in[-2,1]$ while $\lambda$  is replaced by $\frac{\lambda}{l^2}$.

 Note that $$|X|_{\tilde{g}_t}=l|X|_g\leq \frac{2Cl}{\lambda\sqrt{s}}=\frac{2C\xi(n)l^2}{\lambda}.$$
Applying  Lemma \ref{R-estimate} to $\tilde g_t$,  we have
\begin{align}
&|\text{Ric}(\tilde{g})-\frac{\lambda}{l^2}\tilde{g}-\frac{\lambda}{l^2}L_X\tilde{g}|_{\tilde{g}}(x,0)\notag\\
&\leq C(n,A)\{\int_{-2}^{1} dt\int|R(\tilde{g})-n\frac{\lambda}{l^2}-\frac{\lambda}{l^2}tr_{\tilde{g}}(L_X\tilde{g})| \omega_{\tilde g_t}^n\}^{\frac{1}{2}}.\notag
\end{align}
Observe that
\begin{align}
\nonumber &|\text{Ric}(g)-\lambda  g-\lambda L_X  g|_{ g}(x,s)
=l^2|\text{Ric}(\tilde{g})-\frac{\lambda}{l^2}\tilde{g}-\frac{\lambda}{l^2}L_X\tilde{g}|_{\tilde{ g}}(x,0)
\end{align}
and
\begin{align}\nonumber  &\int_{s-l^{-2}}^{s+2l^{-2}} dt \int|R-n\lambda-\lambda \Delta \theta| \omega_{g_t}^n\\
&=l^{-n}\int_{-2}^{1} dt\int|R(\tilde{g})-n\frac{\lambda}{l^2}-\frac{\lambda}{l^2}tr_{\tilde{g}}(L_X\tilde{g})| \omega_{\tilde g_t}^n.\notag
\end{align}
Thus we get
\begin{align}
&|\text{Ric}(g)-\lambda g-\lambda L_X g|_{g}(x,s)\notag\\
&\leq C(n,A)s^{-\frac{n+2}{2}}\{\int_{s-l^{-2}}^{s+2l^{-2}} dt\int|R-n\lambda-\lambda \Delta \theta| \omega_{g_t}^n\}^{\frac{1}{2}},\notag
\end{align}
which implies (\ref{ricci-estimate}).
\end{proof}

\vskip3mm

\section{Estimate for the distance functions}

We are going to compare the distance functions  between  the initial metric  $g_0$ and $g_{\delta}$ in the flow (\ref{rescaling-equation-1}). The following lemma is due to Perelman for  the normalized Ricci flow [Pe].

\begin{lem}\label{distance-derivative}
Let  $g_t=g(\cdot, t)$ $(0\leq t\leq 1)$  be  a  solution  of rescaled Ricci flow on $M^n$ (in our case, $M$ is K\"ahler),
\begin{align}\label{normalized-flow}\frac{d}{dt}g=-\text{Ric}(g) + \lambda g, ~g(0,\cdot)=g_0,\end{align}
where  $0<\lambda \leq 1$.  Let  $x_1, x_2$  be  two points in $M$.   Suppose  that at time $t\geq0$,
 $$\text{Ric}(g_t)(x)\leq (2n-1)K,~
\forall x\in  B_{x_1}(r_0, g_t)\cup B_{x_2}(r_0, g_t)$$
 for some $r_0>0$.   Then
\begin{align}\label{perlman-ode}
\frac{d}{dt}d_{g_t}(x_1,x_2)\geq \lambda d_{g_t}(x_1,x_2)-2(2n-1)(\frac{2}{3}Kr_0+r_0^{-1}).
\end{align}
\end{lem}

\begin{proof}
Without loss of generality,  we may assume that  $t=0$. Putting $$\tilde g_t=(1-\lambda t)g(\frac{\log(1-\lambda t)}{-\lambda}),
(0\leq t< \frac{1}{\lambda}),$$
then $\tilde g=\tilde g_t$ satisfies the Hamilton Ricci flow,
$$\frac{\partial}{\partial t}\tilde{g}=-\text{Ric}(\tilde g).$$
Since  $\tilde g_0=g_0$,    by applying  Lemma 8.3 in [Pe], we have
\begin{align}
\frac{d}{dt}d_{\tilde g_t}|_{t=0}\geq -2(2n-1)(\frac{2}{3}Kr_0+r_0^{-1}).\notag
\end{align}
 Note that
\begin{align}
\tilde{d}_t=-\lambda d+d_t.\notag
\end{align}
Hence   (\ref{perlman-ode})  is true.
\end{proof}

By Lemma \ref{distance-derivative} together with Lemma \ref{tensor} In Section 2,  we give a lower bound estimate for the distance functions along the flow as follows.

\begin{prop}\label{lower-bound-distance}
Under the assumption of Lemma \ref{tensor}, we have that for two points $x_1,x_2$ in $B_q(\frac{1}{2},g_0),$
\begin{align}
d_{g_\delta}(x_1,x_2)\geq d_{g_0}(x_1,x_2)-\frac{C_0}{\lambda}(\sqrt{t}+t^{-\frac{n}{2}}E^{\frac{1}{2}}),~\forall~ t\in (0,\delta],
\end{align}
where  $C_0$ is a uniform constant and   $E=\int_0^{2\delta} dt\int_M|R-\lambda\Delta\theta-n\lambda|\omega_{g_t}^n$.
In particular, when $E\leq \delta^{n+1}$,
\begin{align}\label{distance-integral}
d_{g_\delta}(x_1,x_2)\geq d_{g_0}(x_1,x_2)-\frac{C_0}{\lambda}E^{\frac{1}{2(n+1)}}.
\end{align}
\end{prop}

\begin{proof}
Let $\Phi(t)$ be  a  one parameter subgroup generated by $\text{real}(X)$.  Then $\hat{g}_t=\Phi(-t)^*g_t$  is a solution of
of the normalized flow (\ref{normalized-flow}). Applying  Lemma  \ref{modified-pseudo}  for two points
$y_1=\Phi(-t)x_1$ and $y_2=\Phi(-t)x_2$ by choosing $r_0=\sqrt{t}$,   together with  Lemma \ref{distance-derivative} we have
$$\frac{d}{dt}d_{\hat{g}_t}(y_1,y_2)\geq \frac{\lambda}{2}d_{\hat{g}_t}(y_1,y_2)-C_1t^{-\frac{1}{2}}.$$
   It follows
$$d_{\hat{g}_t}(y_1,y_2)\geq d_{\hat{g}_0}(y_1,y_2)- 2C_1\sqrt{t}.$$
As a consequence, we derive
\begin{align}\label{two-points-distance-1}
d_{g_t}(x_1,x_2)&=d_{\hat{g}_t}(y_1, y_2)\notag
\\
&\ge d_{g_0}(y_1,y_2)-2C_1\sqrt{t}\notag\\
&\ge d_{g_0}(x_1,x_2) -2\|X\|_{g_0} t-2C_1\sqrt{t}\notag\\
&\ge d_{g_0}(x_1,x_2) -\frac{C_2\sqrt{t}}{\lambda}.
\end{align}
On the other hand,  integrating  (\ref{rescaling-equation-1}),  we get  from  Lemma \ref{tensor},
\begin{align}\label{two-points-distance-2}
&\log \frac{d_{g_\delta}(x_1,x_2)}{d_{g_t}(x_1,x_2)}\notag\\
&\geq  -C_3\int_t^\delta s^{-\frac{n+2}{2}}E^{\frac{1}{2}}ds
\geq -C_3'E^{\frac{1}{2}}t^{-\frac{n}{2}},~\forall t>0.
\end{align}
Hence combining (\ref{two-points-distance-1}) and (\ref{two-points-distance-2}),   we obtain
\begin{align}
&d_{g_\delta}(x_1,x_2)\geq d_{g_t}(x_1,x_2)e^{-C_3'E^{\frac{1}{2}}t^{-\frac{n}{2}}}\notag\\
&\geq (d_{g_0}(x_1,x_2)-\frac{C_2\sqrt{t}}{\lambda})e^{-C_3'E^{\frac{1}{2}}t^{-\frac{n}{2}}}\notag\\
&\geq d_{g_0}(x_1,x_2)-\frac{C_4}{\lambda}(\sqrt{t}+t^{-\frac{n}{2}}E^{\frac{1}{2}}).\notag
\end{align}
When $E\leq \delta^{n+1}$,  we can choose $t=E^{\frac {1}{n+1}}$ to get (\ref{distance-integral}).
\end{proof}

Next we  use the above proposition to give an upper  bound  estimate for the distance functions
 by  using a covering argument  as in [TW].

\begin{lem}\label{distant-upper-bound-lemma}
Let  $(M,g(\cdot,t),q)$  $(0\leq t\leq 1)$  be  a solution  of (\ref{rescaling-equation-1})
as in Lemma \ref{tensor}.  Let
$\Omega=B_q(1,g_0),\Omega'=B_q(\frac{1}{2},g_0).$   For  every $l<\frac{1}{2}$,
  we define
\begin{align}
A_{+,l}=\sup_{B_x(s,g_0)\subset \Omega', s\leq l}c_n^{-1}s^{-2n}\text{vol}_{g_0}(B_x(s,g_0))\notag
\end{align}
and
\begin{align} A_{-,l}=\inf_{B_x(s,g_\delta)\subset \Omega',s\leq l}c_n^{-1}s^{-2n}\text{vol}_{g_\delta}(B_x(s,g_\delta)).\notag
\end{align}
Then for any  $x_1,x_2 \in \Omega''=B_q(\frac{1}{4},g_0)$,  it holds
\begin{align}\label{distance-upper-bound}
d_{g_\delta}(x_1,x_2)\leq r+\frac{C_0}{\lambda}A_{+,4r}\{|\frac{A_{+,r}}{A_{-,r}}-1|^{\frac{1}{2n}}+r^{-\frac{1}{2n}}E^{\frac{1}{4n(n+1)}}\}r,
\end{align}
where  $r=d_{g_0}(x_1,x_2)\leq \frac{1}{8}$ and $E<<r^{2(n+1)}.$
\end{lem}

\begin{proof} By  Proposition \ref{lower-bound-distance}, we  see that $$B_{x_1}(r-\frac{C_0}{\lambda}E^{\frac{1}{2(n+1)}},g_\delta)\subset B_{x_1}(r,g_0),$$
where $C_0$ is the constance determined in  (\ref{distance-integral}).
Then
 \begin{align}\label{volume-estimate-3.1} A_{-,r}(r-\frac{C_0}{\lambda}E^{\frac{1}{2(n+1)}})^{2n}\leq \text{vol}_{g_\delta}(B_{x_1}(r-C_0E^{\frac{1}{2(n+1)}},g_\delta)).
\end{align}
Let   $s_0$  be  the largest radius $s$ among all the balls  $B_x(s,g_0)$  such that
$$B_x(s,g_0)\subset B_{x_1}(r,g_0)~\text{and}~ B_x(s,g_0)\cap B_{x_1}(r-\frac{C_0}{\lambda}E^{\frac{1}{2(n+1)}},g_\delta)=\emptyset.$$
Since the volume element $d\text{vol}(g_t)$ satisfies
$$\frac{d}{dt}d\text{vol}(g_t)=(-R+n\lambda+\lambda \Delta \theta) d\text{vol}(g_t),$$
it is easy to see that there is a ball  $B_{x_0}(s_0,g_0)$  such that
\begin{align}\label{volume-comparison}
&\text{vol}_{g_\delta}(B_{x_0}(s_0,g_0))\notag\\
&\leq \text{vol}_{g_\delta}(B_{x_1}(r,g_0))-\text{vol}_{g_\delta}(B_{x_1}(r-\frac{C_0}{\lambda}E^{\frac{1}{2(n+1)}},g_\delta))\notag\\
&\leq  \text{vol}_{g_0}(B_{x_1}(r,g_0))-\text{vol}_{g_\delta}(B_{x_1}(r-\frac{C_0}{\lambda}E^{\frac{1}{2(n+1)}},g_\delta))+E.
\end{align}
Observe that
  $$B_{x_0}(s_0,g_0)\supseteq B_{x_0}(s_0-\frac{C_0}{\lambda}E^{\frac{1}{2(n+1)}},g_\delta).$$
 we have
\begin{align}
 A_{-,r}(s_0-\frac{C_0}{\lambda}E^{\frac{1}{2(n+1)}})^{2n}&\leq \text{vol}_{g_\delta}(B_{x_0}(s_0-\frac{C_0}{\lambda}E^{\frac{1}{2(n+1)}},g_\delta))\notag\\
 &\le \text{vol}_{g_\delta}(B_{x_0}(s_0,g_0)).\notag
 \end{align}
Thus plugging  the above inequality  into (\ref{volume-comparison}) together with (\ref{volume-estimate-3.1}) and the fact that
 \begin{align} \text{vol}_{g_0}(B_{x_1}(r,g_0))\leq A_{+,r}r^{2n},\notag\end{align}
we obtain
\begin{align}\label{radius-estimate} s_0\leq \{|\frac{A_{+,r}}{A_{-,r}}-1|+\frac{C_0}{\lambda}r^{-1}E^{\frac{1}{2(n+1)}}\}^{\frac{1}{2n}}r+\frac{C_0}{\lambda}E^{\frac{1}{2(n+1)}}.
\end{align}

On the other hand,   since
$$B_{x_2}(3s_0,g_0)\cap B_{x_1}(r-\frac{C_0}{\lambda}E^{\frac{1}{2(n+1)}},g_\delta)\neq\emptyset,$$
 we see that  there exists some  point
  $$x_3\in B_{x_2}(3s_0,g_0)\cap B_{x_1}(r-\frac{C_0}{\lambda}E^{\frac{1}{2(n+1)}},g_\delta).$$

  \begin{claim}\label{claim-section-3}  There is a uniform  constant $C_1=C_1(n)$ such that
\begin{align}\label{claim-small-ball}
d_{g_\delta}(x_2,x_3)\leq C_1A_{+,4r}max\{s_0,\frac{3C_0}{\lambda}E^{\frac{1}{2(n+1)}}\}.
\end{align}
\end{claim}

Combining  (\ref{claim-small-ball})
 with  (\ref{radius-estimate}),  we  will finish the proof of  (\ref{distance-upper-bound})
 because of  the  triangle inequality
$$d_{g_\delta}(x_1,x_2)\leq d_{g_\delta}(x_1,x_3)+d_{g_\delta}(x_2,x_3).$$

To prove Claim \ref{claim-section-3},  we  first assume  that
\begin{align}\label{assumption-3.1} s_0>\frac{3C_0}{\lambda}E^{\frac{1}{2(n+1)}}.
\end{align}
 Let $\gamma$ be the  minimizing  geodesic curve   which connecting  $x_2$ and  $x_3$ in $(M,g_0)$. Choose $N$  geodesic balls $B_{z_i}(s_0,g_\delta)$  in $(M,g_\delta)$  such that $B_{z_i}(\frac{s_0}{2},g_\delta)$ are disjoint.
Since
\begin{align}
&B_{z_i}(\frac{r_0}{2},g_\delta)\subset B_{z_i}(\frac{s_0}{2}+\frac{C_0}{\lambda}E^{\frac{1}{2(n+1)}},g_0 )\notag\\
&\subset B_{z_i}(s_0,g_0)
 \subset B_{x_2}(4s_0, g_0)\subset B_{x_1}(\frac{1}{2},g_0),\notag
\end{align}
we have
\begin{align}
\nonumber N A_{-,r}(\frac{s_0}{2})^{2n}&\leq \sum_{i=1}^N\text{vol}_{g_\delta}(B_{z_i}(\frac{s_0}{2},g_\delta))
\leq  \text{vol}_{g_\delta}( B_{x_2}(4s_0))&\\
\nonumber &\leq \text{vol}_{g_0}B_{x_2}(4s_0)+E\leq A_{+,4r}(4s_0)^{2n}+E.&
\end{align}
Noticing  that  by the Bishop volume comparison and Lemma \ref{modified-pseudo}, we see that
$$A_{-,r}\geq C(n,\delta)=C(n).$$
By (\ref{assumption-3.1}), it follows
$$N\le C'  A_{+,4r}. $$
 Since
 $$d_{g_\delta}(x_2,x_3)\leq 2N s_0,$$
 we  deduce (\ref{claim-small-ball})  from (\ref{radius-estimate})  immediately.

Secondly,  we assume that
$$s_0\leq \frac{3C_0}{\lambda}E^{\frac{1}{2(n+1)}}.$$
 In  this case,   we can  copy  the above argument  of  geodesic balls covering
to  prove   (\ref{claim-small-ball})   while  the radius $s_0$ of balls is  replaced by   $\frac{3C_0}{\lambda}E^{\frac{1}{2(n+1)}}$. The claim is proved.
\end{proof}

\begin{prop}\label{Gromov-distance}
Let  $(M, g(\cdot, t),q)$  $(0\leq t\leq 2\delta)$  be a solution of (\ref{rescaling-equation-1})
as in Lemma \ref{tensor}.
Then for two points $x_1,x_2 \in \Omega''=B_q(\frac{1}{4},g_0)$  with  $r=d(x_1,x_2,g_0)\leq \frac{1}{8}$,  we have
\begin{align}\label{upper-distant-3}
 d(x_1,x_2,g_\delta)\leq r+\frac{C_0}{\lambda}E^{\frac{1}{6n(n+1)}}r,
\end{align}
if $E<<r^{6(n+1)}.$
\end{prop}

\begin{proof}
By the Bishop volume comparison and   Lemma \ref{modified-pseudo},  we  see that
 $$A_{-,r}\geq 1-A r,$$
 for some uniform constant $A$, where $r\le \delta<<1.$
 Also by  the  volume comparison in [WW], we have
 $$A_{+,r}\leq 1+Ar^2,~\forall~ r\le 1.$$
  Applying  Lemma  \ref{distant-upper-bound-lemma} to any  two points $x_1,x_2 \in \Omega''$ with $d_{g_0}(x_1,x_2)=r\le \delta <<1$,  we get
\begin{align}\label{small-distance}
d_{g_\delta}(x_1,x_2)r^{-1}\leq 1+ \frac{C_0}{\lambda}(r^\frac{1}{n}+r^{-\frac{1}{2n}}E^{4n(n+1)}).
\end{align}
For  general two points $x_1,x_2$ with $d(x_1,x_2,g_0)=l\le \frac{1}{8}$,  we divide the minimal geodesic curve   which connecting
 $x_1$  and  $x_2$ into $N$ parts  with the  same  length $\frac{l}{N}\le \delta$.    Thus   by (\ref{small-distance}),
we obtain
 $$\frac{d(x_1,x_2,g_\delta)}{N^{-1}l}\leq N\{1+\frac{C_0}{\lambda}\{(N^{-1}l)^\frac{1}{n}+(N^{-1}l)^{-\frac{1}{2n}}E^{4n(n+1)}\}\}.$$
Choosing $N\sim lE^{-\frac{1}{6(n+1}}$, we  derive (\ref{upper-distant-3}).
\end{proof}

\vskip3mm

\section{Almost K\"{a}hler Ricci solitons}

In this section, we are able  to prove  the smoothness  of   the   regular part of the  limit space for a sequence of   weak  almost K\"ahler-Ricci solitons studied in  [WZ].      Recall the definition of  weak  almost K\"ahler-Ricci solitons.

\begin{defi}\label{almost-kr-soliton}    We call a sequence of  K\"ahler metrics $\{(M_i,g^i,J_i)\}$
  weak  almost K\"ahler-Ricci solitons
  if  there are uniform constants $\Lambda$ and $A$  such that
\begin{align}
& i)~  \text{Ric}(g^i)+ L_{X_i}  g^i\ge -\Lambda^2 g^i, ~\text{im}(  L_{X_i}  g^i)=0;\notag\\
&ii)~ |X_i|_{g^i}\le A;\notag\\
&iii) ~\text{lim}_{i\to\infty}\|\text{Ric}(g^i)-g^i+L_{X_i} g^i\|_{L^1_{M_i}(g^i)}= 0.\notag
\end{align}
Here $\omega_{g^i}\in 2\pi c_1(M_i, J_i)$ and  $X_i$ are   reductive holomorphic vector fields
  on   Fano manifolds $(M_i,J_i)$.
 \end{defi}

We now assume that
 \begin{align}\label{further-condition-1}\text{vol}_{g^i}(B_{p^i}(1))\ge  v>0,~\text{for some}~p^i\in M_i.\end{align}
Let  $g^i_t=g^i(\cdot, t)$ be a solution of  the  K\"ahler-Ricci flow (\ref{kr-flow})  on $(M_i,J_i)$ with $g^i$  the  initial metric.  Suppose that  $g^i_t$ satisfies
\begin{align}\label{further-condition-2}
 |X_i|_{g^i_t}\leq \frac{B}{\sqrt{t}}\end{align}
 and
 \begin{align}\label{further-condition-3}
  \int_0^1dt\int_{M_i} |R(g^i_t)-\Delta\theta_{g^i_t}-n| \omega_{g^i_t}^n\rightarrow 0, ~\text{as}~i\to\infty.
\end{align}
Here $B$ is a uniform constant.   We note that (\ref{further-condition-2}) and (\ref{further-condition-3}) have been  used in Lemma \ref{tensor},
Proposition \ref{lower-bound-distance}  and Proposition \ref{Gromov-distance}, respectively. Under the assumption   (\ref{further-condition-1})-(\ref{further-condition-3}), we prove

\begin{theo}\label{main-theorem-1}
Let $\{(M_i,g^i,J_i\}$ be a sequence of
  weak  almost K\"ahler-Ricci solitons.  Suppose that $g^i$ satisfy the conditions (\ref{further-condition-1})-(\ref{further-condition-3}). Then there exists a subsequence of   $\{g^i\}$ which converge to a  K\"ahler-Ricci soliton   with complex codimension of singularities at least 2
  in  the Gromov-Hausdorff topology.\end{theo}

\begin{proof} It was proved in [WZ] that under the   condition (\ref{further-condition-1}) there exists a subsequence of   $\{g^i\}$ which converge to a  metric space $(Y, g_\infty)$  with complex codimension of singularities   of $Y$ at least 2.  Denote $\mathcal{R}$   as the regular part of $Y$. We  want to show that   $\mathcal{R}$    is an open manifold and $g_\infty$ is in fact a  K\"ahler-Ricci soliton  for some complex structure on $\mathcal{R}$.

 Let $y_0\in \mathcal{R}$.   This   means that the tangent cone $T_{y_0}$ at $y_0$ is isometric to $\mathbb{R}^{2n}$.   Then by the Volume Convergence Theorem 4.10 in [WZ], it is easy to see that for any $\delta>0$ there exists $r_0<<1$ such that
\begin{align}
\text{vol}(B_{y_0}(r))> (1-\delta)c_nr^{2n},~\forall ~r<r_0.\notag
\end{align}
Again by  the  above convergence theorem together with  the monotonicity of volume [WW],  there exists an $\epsilon>0$  such that for any $y\in
B_{y_0}(\epsilon, g_\infty)$ it holds
\begin{align}\label{almost-small-eucldean-ball}
\text{vol}(B_y(r))> (1-\delta)c_n r^{2n}, ~\forall ~r<r_0.
\end{align}

\begin{claim}\label{claim} $y\in \mathcal{R}$ for any $y\in
B_{y_0}(\epsilon,g_\infty)$.
\end{claim}

For a fixed $r$,   we choose a sequence of  geodesic balls  $B_{q_i}(r)\subset M_i$  which converge  to $B_y(r)$ in  the Gromov-Hausdorff topology.  Then
 by (\ref{almost-small-eucldean-ball}), for  $i$  large enough,
we have
\begin{align}\label{almost-small-eucldean-ball-sequence}
\text{vol}(B_{q_i}(r))> (1-\delta)c_n r^{2n}.
\end{align}
 Scale  $g^i$ to $\hat g^i=\frac{1}{r} g^i$ and we consider the solution $\hat g^i(\cdot,t)=\hat g^i_t$ of  flow (\ref{rescaling-equation-1}) with the initial metric $\hat g^i$, where $\lambda=r$.
By applying Proposition \ref{lower-bound-distance} and   Proposition  \ref{Gromov-distance} to each  ball  $B_{q_i}(1,\hat g^i )$,  we obtain
\begin{align}\label{distance-estimate-5}|d_{\hat g^i}(x_1,x_2)-d_{\hat g^i_\delta}(x_1,x_2)|\leq C E^{\frac{1}{6n(n+1)}}, ~\forall~ x_1, x_2\in B_{q_i}(\frac{1}{4},\hat g^i),
\end{align}
where
 $$E=\frac{1}{r^{n-1}}\int_0^{2\delta} dt\int_M|R(g^i_t)-\lambda\Delta\theta_{g^i_t}-n\lambda|\omega_{g^i_t}^n\to 0,~\text{as}~i\to\infty.$$
 On the other hand,  since  the curvature are  uniformly bounded in $B_{q_i}(1,$$\hat g^i_\delta)$ by  Lemma \ref{modified-pseudo},  $B_{q_i}(1,\hat g^i_\delta)$ converge to a   smooth metric ball
 $B_{y\infty}(1, $$ \hat g_\infty')$ by the regularity of $\hat g^i_\delta$.   Hence  by (\ref{distance-estimate-5}), we derive
  \begin{align}
s^{-1}d_{GH}(B_y(s,g_\infty),  B_{y\infty}(s,g_\infty'))\leq Ls^{2}, ~\forall  ~s\leq \frac{r}{4}.
\end{align}
 where $L$ is a uniform constant and  $g_\infty'=r \hat g_\infty$.
This means that the tangent cone at $y$ is isometric to $\mathbb{R}^{2n}$, so the claim is proved.

By the above claim,   we see that  there exists a small $r$ for any $y\in \mathcal{R}$  such that  $B_{y}(r)\subset \mathcal{R}$ and (\ref{almost-small-eucldean-ball})
is satisfied.  Then following the argument in the proof of Claim \ref{claim},  there exists a sequence of  geodesic balls $(B_{q_i}(r),g^i_\delta)\subset M_i$
which converge  to $B_y(r)$ in $C^\infty$-topology.  Consequently,  the potentials $\theta_{g^i_\delta}$ of $X_i$ restricted on  $(B_{q_i}(r),g^i_\delta)$ converge to a smooth  function  $\theta_\infty$  defined on  $B_y(r)$. Namely,
 $$\lim_{i\to\infty} \Psi_i^*(\theta_{g^i_\delta})=\theta_\infty,$$
where $\Psi_i$ are diffeomorphisms from $B_y(r)$ to  $B_{q_i}(r)$   such that $\Psi_i^*(g^i_\delta) $  converge to $g_\infty$ and $\Psi_i^* J_i$  converge  to some  limit complex structure $J_\infty$ on $B_y(r)$.
By the regularity  of  flow (\ref{rescaling-equation-1}) and the condition (\ref{further-condition-3}),  $\theta_\infty$  satisfies in $B_y(r)$,
\begin{align}\label{limit-equation}
\Delta\theta_{g_\infty}=R(g_\infty)-n~\text{and}~\partial\partial\theta_\infty=0.\end{align}
Moreover, by  (\ref{ricci-estimate}) in Lemma \ref{tensor}, we get
\begin{align}\label{ricci-limit} \text{Ric}(g_\infty)- g_\infty- \sqrt{-1}\partial\overline\partial\theta_\infty=0,~\text{in}~B_y(r).
\end{align}
Hence,  $\theta_\infty$ can be extended to a potential of  holomorphic vector field $X_\infty$ on $(\mathcal{R}, J_\infty)$,
and  consequently $g_\infty$ is a  K\"ahler-Ricci soliton on
$\mathcal{R}$.
\end{proof}

\begin{rem} It seems that the limit space $Y$  in Theorem \ref{main-theorem-1} is actually a normal algebraic variety as showed  in  recent papers   by Tian,  Chen-Donaldson-Sun  to solve
the  Yau-Tian-Donaldson
conjecture for K\"ahler-Ein-stein metrics [T2], [CDS].
\end{rem}

In [WZ],   it was  showed that   there exists  a sequence of
  weak  almost K\"ahler-Ricci solitons  $g^s$ $(s<1)$ on  a Fano manifold $(M, g, J)$ if  the modified K-energy $\mu(\cdot)$ is bounded below.  Here $\mu(\cdot)$ is defined for any  $K_X$-invariant K\"{a}hler potential $\phi$ by ([TZ2]),
\begin{align}\mu(\phi)&=-\frac{n}{V}\int_0^1\int_M\dot{\psi} [\text{Ric }(\omega_\psi)-\omega_\psi-\sqrt{-1}\partial\bar{\partial}\theta_{\omega_\psi}  \notag\\
&+\sqrt{-1}\bar{\partial}(h_{\omega_\psi}-\theta_{\omega_\psi}) \wedge\partial\theta{\omega_\psi}]
\times e^{\theta_{\omega_\psi}}\omega^{n-1}_\psi\wedge dt.\notag\end{align}
In fact,  such  $g^s$  are  a family of K\"{a}hler metrics induced by   the K\"{a}hler potential  solutions $\phi_s$ of a family of
complex Monge-Amp\`ere equations, which are equivalent  to a family of Ricci curvature equations,
\begin{align}\label{continuity-equation}
\text{Ric }(\omega_{\phi_s})=s\omega_{\phi_s}+(1-s)\omega_g +L_X\omega_{\phi_s}.
\end{align}
(\ref{continuity-equation}) are also equivalent   to equations,
 \begin{align}\label{ricci-potential-equa} h_{\omega_{\phi_s}}
-\theta_{\omega_{\phi_s}}=-(1-s)\phi_{s},\end{align}
where $h_{\omega_{\phi_s}}$ are  the Ricci potentials of  $\omega_{\phi_s}$.

 In the following, we need to  verify the  conditions (\ref{further-condition-2}) and  (\ref{further-condition-3}) for $g^s$.  We note that (\ref{further-condition-1}) is true for $g^s$ [WZ].  Thus as an application of   Theorem \ref{main-theorem-1},  we prove that

\begin{theo}\label{main-theorem-2}  There exists  a sequence  of
  weak  almost K\"ahler-Ricci solitons   $\{g^{s_i}\}$  $(s_i\to 1)$   which converge to a  K\"ahler-Ricci soliton  with complex codimension of singularities at least 2   in the Gromov-Hausdorff topology.  In the other words,  a Fano manifold with   the modified K-energy  bounded below can be deformed to a  K\"ahler-Ricci soliton  with complex codimension of singularities at least 2.

\end{theo}

\begin{lem}
\begin{align}\label{ricci-potential-estimate-1} h_{g_{s}}-\theta_{g_s}\rightarrow 0,~\text{as}~s\to 1.\end{align}
Consequently,
 \begin{align}\label{ricci-potential-estimate-2} |h_{g_{s}}|\leq C.\end{align}
\end{lem}

\begin{proof}
Recall  the  two functionals $I$ and $J$ defined for  $K_X$-invariant K\"{a}hler potentials by ([Zh], [TZ1]),
 \begin{align}
 I(\phi)=\int_M\phi(e^{\theta_{\omega_0}}\omega_0^n-e^{\theta_{\omega_\phi}}\omega_\phi^n)\notag
 \end{align}
 and
 \begin{align}J(\phi)=\int_0^1\int_M\dot{\phi_t}(e^{\theta_{\omega_0}}\omega_0^n-e^{\theta_{\omega_\phi}}\omega_\phi^n)dt.\notag
 \end{align}
It was showed  for the potential $\phi_s$  in  [TZ1] that
$$-\frac{d}{ds}\mu(\phi_s)=(1-s)\frac{d}{ds}(I-J)(\phi_s).$$
 Then
 $$(I-J)(\phi_s)=-\frac{\mu(\phi_s)}{1-s}+\int_0^s\frac{\mu(\phi_\tau)}{(1-s)^2}ds .$$
 Since $\mu(\phi_s)$ is monotone and bounded below,  $\lim_{s\rightarrow 1^-}\mu(\phi_s)$ exists.  By 'H$\hat o$pital's rule,  it is easy to see that
  $$\lim_{s\rightarrow 1^-} (1-s)\int_0^s\frac{\mu(\phi_\tau)}{(1-\tau)^2}d\tau=\lim_{s\rightarrow 1^-}\mu(\phi_s).$$
  Thus we get
  $$\lim_{s\rightarrow 1^-}(1-s)(I-J)(\phi_s)=0.$$
  On the other hand, by using the Green formula [Ma] (also see [CTZ]), there exists a uniform constant $C$ such that
  $$\text{osc}(\phi_s)\le \|\phi_s\|_{C^0}\le I(\phi_s)+C .$$
  It follows  that
  $$(1-s)\|\phi_s\|_{C^0}\le (1-s) (c(I-J)(\phi_s)+C)\to 0.~\text{as}~s\to 1.$$
 Hence  by  (\ref{ricci-potential-equa}),  we obtain   (\ref{ricci-potential-estimate-1}).
  (\ref{ricci-potential-estimate-2}) is a direct  consequence  of   (\ref{ricci-potential-estimate-1})  since $\theta_{g_s}$ are uniformly bounded [Zh].
 \end{proof}

\begin{lem}\label{flow-estimate}
Let  $g^s_t=g^s(\cdot, t)$ be a solution of  the  K\"ahler-Ricci flow (\ref{kr-flow}) with the above $g^s$ as an initial metric.
Then
\begin{align}\label{further-condition-1'}
 |X|_{g^s_t}\leq \frac{B}{\sqrt{t}}.\end{align}
 \end{lem}

\begin{proof}
Let $u_t$ be  the K\"ahler potential of $g^s_t$.  Namely,  it  is defined  by
$$\omega_{g^s_t}=\omega_{g^s}+\sqrt{-1}\partial\bar{\partial}u.$$
According to Lemma 4.3 in [CTZ],  we have
$$|\nabla(\frac{\partial}{\partial t}u)|_{g^s_t }\leq e^2\frac{\|h_{g^s}-\theta_{g^s}\|_{C^0}}{\sqrt{t}},0< t\leq 1.$$
Since $\tilde {g^s_t}=\Phi_t^*(g^s_t)$  is a solution of  the K\"ahler-Ricci flow,
\begin{align}\label{flow}
\frac{\partial}{\partial t} g=-\text{Ric}(g)+g\notag,
\end{align}
where $\Phi(-t)$  is   an  one parameter subgroup generated by $\text{real}(X)$,  we also  have for the  K\"ahler potential $\tilde u$ of $\tilde {g^s_t}$ ([T1]),
$$|\nabla(\frac{\partial}{\partial t}\tilde{u})|_{\tilde {g^s_t}}\leq e^2\frac{\|h_{g^s}\|_{C^0}}{\sqrt{t}},~\forall~0< t\leq 1.
$$
Note that
\begin{align}
\frac{\partial}{\partial t}u=\Phi_t^*(\frac{\partial}{\partial t}\tilde{u})+\theta_{g^s_t}+m(t).\notag
\end{align}
We get
\begin{align}
|X|_{g^s_t}=|\nabla \theta_{g^s_t}|_{g^s_t}\leq |\nabla(\frac{\partial}{\partial t}\tilde{u})|_{\tilde {g^s_t}}+|\nabla(\frac{\partial}{\partial t}u)|_{g^s_t}.\notag
\end{align}
Now (\ref{further-condition-1'}) follows from (\ref{ricci-potential-estimate-2}) immediately.
\end{proof}

\begin{lem}\label{further-condition-2'}
Let  $g^s_t=g^s(\cdot, t)$ be a solution of  the K\"ahler-Ricci flow as in Lemma \ref{flow-estimate}.  Then
\begin{align}
  \int_0^1 dt \int_{M} |R(g^s_t)-\Delta\theta_{g^s_t}-n|  \omega_{g^s_t}^n\rightarrow 0, ~\text{as}~s\to 1.
  \end{align}
  \end{lem}

\begin{proof} First by (\ref{continuity-equation}), we note that
 \begin{align}
  (\Delta+X)(h_{g^s}-\theta_{g^s})\ge -(1-s)n - (1-s)|X(\phi_s)|\ge -(1-s)(c_1+n),\notag\end{align}
 where $c_1=\sup\{\|X(\phi)\|_{C^0(M)}|~K_X-\text{invariant K\"ahler potential}~\phi \}$ is a bounded  number [Zh].
  By the Maximum Principle,  it follows  that (cf. Lemma 4.2 in [CTZ]),
  \begin{align}
   (\Delta+X)(h_{g^s_t}-\theta_{g^s_t}) \ge -(1-s)(c_1+n)  e^t, ~\forall~0< t.
   \notag\end{align}
The above implies  that (cf. Lemma 4.4 in [CTZ]),
\begin{align}
&\int_M|\nabla (h_{g^s_t}-\theta_{g^s_t})|^2 e^{\theta_{g^s_t}}   \omega_{g^s_t}^n  \notag\\
&\le 2e^2(c_1+n) (1-s)\|h_{g^s}-\theta_{g^s}\|_{C^0(M)}, ~\forall~0< t\leq 1.\notag
\end{align}
Hence by (\ref{further-condition-1'}),
we get
\begin{align}
\int_0^1dt\int_M|X(h_{g^s_t}-\theta_{g^s_t})|e^{\theta_{g^s_t}} \omega_{g^s_t}^n
&\leq \int_M|\nabla( h_{g^s_t}-\theta_{g^s_t} )| e^{\theta_{g^s_t}}\omega_{g^s_t}^n \int_0^1|X|_{g^s_t}dt\notag\\
&\leq  C(1-s)^{\frac{1}{2}} \int_0^1\frac{1}{\sqrt{t}} dt \rightarrow 0, ~\text{as}~s\to 1.\notag\end{align}
Therefore,
\begin{align}
&  \int_0^1 dt \int_{M} |R(g^s_t)-\Delta\theta_{g^s_t}-n|   e^{\theta_{g^s_t}} \omega_{g^s_t}^n\notag\\
&\leq \int_0^1dt\int_M| \Delta(h_{g^s_t}-\theta_{g^s_t})+X(h_{g^s_t}-\theta_{g^s_t})
+ (1-s)(c_1+n)|e^{\theta_{g^s_t}} \omega_{g^s_t}^n
\notag\\
&+ \int_0^1dt\int_M |X(h_{g^s_t}-\theta_{g^s_t})|
e^{\theta_{g^s_t}} \omega_{g^s_t}^n \notag +V  (1-s)(c_1+n) \\
 &=\int_0^1dt\int_M|X(h_{g^s_t}-\theta_{g^s_t})|e^{\theta_{g^s_t}} \omega_{g^s_t}^n  + 2V  (1-s)(c_1+n) \notag\\
\nonumber &\rightarrow 0,~\text{as}~s\to 1.\notag
\end{align}
This finishes the proof of Lemma  \ref{further-condition-2'}.
\end{proof}


\begin{thebibliography}{99}

\nonumber\bibitem{CC}[CC] Cheeger, J. and Colding, T.,  On the structure of spaces with Ricci curvature bounded below I, J. Differential Geom. 45 (1997), 406-480.

\nonumber\bibitem{CDS}[CDS] Chen, X., Donaldson, S. and Sun, S., K\"ahler-Einstein metrics on Fano manifolds, III,   arXiv:1302.0282.

\nonumber\bibitem{CGT}[CGT] Cheeger, J.,  Gromov, M. and  Taylor, M.,   Finite propagation speed, kernel estimates for functions of the Laplace operator, and the geometry of complete Riemannian manifolds, J. Differential Geom.,  17 (1982), 15-54.

\nonumber\bibitem{CTZ}[CTZ] Cao, H.,  Tian, G. and  Zhu, X.H.,  K\"{a}hler-Ricci solitons on compact complex manifolds with $c_1(M)>0$, GAFA, Vol. 15 (2005) 697-719.

\nonumber\bibitem{Gr}[Gr]   Gross, L.,  Logarithmic Sobolev inequalities and contractivity properties of semigroups. Dirichlet forms (Varenna, 1992),  Lecture Notes in Math.,  1563 (1993), 54-88.


\nonumber\bibitem{Ma}[Ma]  Mabuchi, T.,  Multiplier hermitian structures on K\"{a}hler manifolds, Nagoya Math. J. 170 (2003), 73-115.

\nonumber\bibitem{Pe}[Pe] Perelman, G.,  The entropy formula for the Ricci flow and its geometric applications, arXiv:math.DG/0211159.

\nonumber\bibitem{Li}[Li] Li, P.,  Lecture notes on geometric analysis, Cambridge Studies in Advanced Mathematics, No. 134 (2012), Cambridge University Press.



\nonumber\bibitem{Ro}[Ro] Rothaus, O.,  Logarithmic Sobolev inequality and the spectrum of Schr\"odinger operators, J. Funct. Analysis, 42 (1981), no. 1, 110-120.

\nonumber\bibitem{T1}[T1]  Tian, G.,   K\"ahler-Einstein metrics with positive scalar curvature, Invent. Math. 130 (1997), 1-37.

\nonumber\bibitem{T2}[T2]  Tian, G., K-stability and K\"ahler-Einstein metrics,  arXiv:1211.4669.

\nonumber\bibitem{TW}[TW]  Tian, G. and Wang, B.,  On the structure of almost Einstein manifolds, arXiv: ~~
\newline math.DG/0211159.1202.2912.

\nonumber\bibitem {TZ1}[TZ1] Tian, G. and Zhu, X.H., Uniqueness of K\"{a}hler-Ricci soltions, Acta Math., 184 (2000),  271-305.

\nonumber\bibitem{TZ2}[TZ2] Tian,  G. and  Zhu, X.H., A new holomorphic invariant and uniqueness of K\"{a}hler-Ricci solitons,
Comment. Math. Helv. 77 (2002), 297-325.

\nonumber\bibitem{TZ3}[TZ3] Tian,  G. and  Zhu, X.H.,  Convergence of K\"ahler-Ricci flow, Journal of the Amer. Math. Soci., 20 (2007), 675-699.


\nonumber\bibitem{WW}[WW] Wei, G.  and Wylie, W., Comparison geometry for Bakry-Emery Ricci curvature,  J. Differential Geom. 83 (2009), 337-405.


\nonumber\bibitem{WZ}[WZ]  Wang, F. and Zhu, X.H., On the structure of spaces with Bakry-\'{E}mery  Ricci curvature  bounded below,  arXiv: math.DG/1304.4490.


\nonumber\bibitem{Zh}[Zh]  Zhu, X.H.,  K\"ahler-Ricci soliton type equations on compact complex manifolds with $C_1(M)>0$,    J. Geom. Anal., 10 (2000),  759-774.
\end{thebibliography}
\end{document}